\documentclass[12 pt]{article}
\usepackage{stmaryrd}
\usepackage{times}
\usepackage{booktabs}
\usepackage{subfigure}
\usepackage{adjustbox}
\usepackage{rotating}
\usepackage{diagbox}
\usepackage{tabularx}
\usepackage{multirow, makecell}
\usepackage{rotating}
\usepackage{longtable}
\usepackage{supertabular}
\usepackage{pifont}
\usepackage{floatrow}
\usepackage{caption}
\usepackage{mathrsfs}
\usepackage[fleqn]{amsmath}
\usepackage{amsfonts,amsthm,amssymb,mathrsfs,bbding}
\usepackage{txfonts}
\usepackage{graphics,multicol}
\usepackage{graphicx}
\usepackage{makecell}
\usepackage{color}
\usepackage{caption}
\usepackage{ textcomp }
\usepackage{ stmaryrd }
\usepackage{indentfirst}
\usepackage{cite}
\usepackage{latexsym,bm}
\usepackage{enumerate}
\evensidemargin=\oddsidemargin\topmargin=-1.5cm

\newtheorem{lem}{Lemma}[section]

\newtheorem{thm}{Theorem}[section]
\newtheorem{cor}{Corollary}[section]

\theoremstyle{definition}

\addtocounter{section}{0}
\begin{document}
\title{Ordering   $Q$-indices  of graphs:  given size and girth \footnote{Supported by National Natural Science Foundation of China (Nos. 12061074, 11971274) and the China Postdoctoral Science Foundation (No. 2019M661398).}}
\author{
{\small Yarong Hu$^{a,c}$, \  \ Zhenzhen Lou$^{b,}$\footnote{Corresponding author.
\it Email addresses:  xjdxlzz@163.com (Z. Lou).},\ \ Qiongxiang Huang$^{a}$}\\[2mm]
\footnotesize $^a$ College of Mathematics and System Science,
Xinjiang University, Urumqi 830046, China \\
\footnotesize $^b$ School of Mathematics, East China University of Science and Technology,
 Shanghai, 200237, China\\
\footnotesize $^c$ School of Mathematics and Information Technology, Yuncheng University, Yuncheng 044000, China}
\date{}
\maketitle {\flushleft\large\bf Abstract:}
The signless Laplacian matrix in graph spectra theory is a remarkable matrix of graphs,
and it is extensively studied by researchers.
In 1981,  Cvetkovi\'{c} pointed $12$  directions in further investigations of graph spectra,
one of which is ``classifying and ordering graphs''.
Along with this classic direction,
we pay our attention on the order of  the largest eigenvalue of the signless Laplacian matrix of graphs, which is usually  called the $Q$-index of a graph. Let $\mathbb{G}(m, g)$ (resp. $\mathbb{G}(m, \geq g)$) be the family of connected graphs on $m$ edges
with girth $g$ (resp. no less than $g$), where $g\ge3$.
In this paper, we  firstly order  the first $(\lfloor\frac{g}{2}\rfloor+2)$ largest  $Q$-indices of graphs in $\mathbb{G}(m, g)$, where  $m\ge 3g\ge 12$.
Secondly, we  order the first $(\lfloor\frac{g}{2}\rfloor+3)$ largest $Q$-indices of graphs in $\mathbb{G}(m, \geq g)$, where $m\ge 3g\ge 12$. As a complement, we   give
the first five largest $Q$-indices of graphs in $\mathbb{G}(m, 3)$ with $m\ge 9$.
Finally, we give the order of the first eleven largest $Q$-indices of all connected  graphs with size $m$.

\begin{flushleft}
\end{flushleft}
\section{Introduction}\label{se-1}
All graphs  considered here are  simple and undirected.   Let $G=(V(G),E(G))$ be a graph with vertex set $V(G)$ and  edge set $E(G)$, where $n(G)=|V(G)|$ denote the \emph{order} and $|E(G)|=m(G)$ the \emph{size} of $G$.
The set of the neighbors of a vertex $v\in V(G)$ is denoted
by $N_{G}(v)$, and  the degree of $v$ is denoted by $d_G(v)$ or $d(v)$.
Let $\Delta=\Delta(G)$  be the maximum degree of $G$.
The \emph{girth} of a graph $G$, denoted by $g$, is the length of the shortest cycle  in $G$.
Let $A(G)$ and $D(G)$ be the adjacency matrix and the diagonal degree matrix of $G$, respectively. The \emph{signless Laplacian matrix} of $G$ is defined as
$Q(G) = D(G) + A(G)$.
The largest eigenvalue of $Q(G)$ is called the \emph{$Q$-index} of $G$, denoted by $q(G)$. Note that $Q(G)$ is a non-negative matrix.
From Perron-Frobenius theorem, there exists a non-negative unit eigenvector
$\mathbf{x}$ corresponding to $q(G)$.
Such eigenvector $\mathbf{x}$  is called the \emph{Perron vector} of $Q(G)$
and
the entry of $\mathbf{x}$ corresponding to vertex $u$ is denoted by $x_u$.
Moreover,
if $G$ is connected then the Perron vector $\mathbf{x}$ is a  positive vector.
As usual,
let $K_{1,n-1}$, $K_n$, $C_n$ and $P_n$ be  respectively  the star, complete graph, cycle and the path of order $n$.

 In 1981, Cvetkovi\'{c} \cite{Cvetkovic-direct} indicated 12 directions in further investigations of graph spectra,
one of which is ``classifying and ordering graphs''. Hence ordering graphs with various properties
by their spectra becomes an attractive topic (see \cite{Chang,Yuan,Yuan-Liu,Lin,
Liu-liu,Wei-Liu}).
The signless Laplacian matrix is a remarkable matrix of graphs, and it is extensively
studied by researchers. Cvetkovi\'{c} and Simi\'{c} presented a series of surveys on the signless
Laplacian spectral theory \cite{Cvetkovic1,Cvetkovic2,Cvetkovic3}.
 Up to now, a simple and
generality method on ordering graphs according to their spectra  (or $Q$-spectra ) has not yet been obtained.

In  particular,  there exists  a significant amount of research on determining  the graph with the first  largest $Q$-index in a giving graph set.
The giving graph set usually contains two aspects:  $H$-free graph set,
for example,  $K_{r+1}$-free, $C_4$-free, $C_6$-free,  $C_k$-free for $k$ is odd,
 and minors-free \cite{turan,zhai-1,zhai-lin-1,Nikiforov-2,M.F,M.C,C.M}.
On the other side, the graph set with giving order and some graph invariant,
such as,  diameter \cite{Feng-L,Wang-J},
  clique number\cite{He-B},
  chromatic number\cite{Yu-G},
  graphic  degree  sequence \cite{zhang-xd} and so on.
One of other graph invariants is the girth,  which has a rich  research.
Given girth and order, the maximum $Q$-index of unicycle, bicycle graph are determined in \cite{wang-uni-bi} and \cite{zhuzx}, one of tricycle, $k$-cyclic graph are determined in    \cite{wang-tri} and  \cite{Liu-mh}, respectively.
Also,
the graph set with giving size and graph invariant,  one can see,
diameter  \cite{Lou-Z},
clique number (resp. chromatic number)  \cite{zhai-xue-lou},
girth (resp. circumference)  \cite{Zhai}, matching number  \cite{zhai-xue}.

All these results as mentioned above   determined the graph with the first largest $Q$-index.
However, there has little progress in the study for the spectral radius ordering problem, and there are few results related to the second largest, third largest, etc.
In \cite{Liu-liu,Yuan,Yuan-Liu,Lin}, the authors compared the  $Q$-indcies of two graphs by comparing their maximum degrees.
In  \cite{Liu-liu}, Liu, Liu and Cheng stated that the ordering of trees according to their $Q$-indices can be transferred to
the ordering of trees with large maximum degree. However, there is no way to compare the $Q$-indices of two graphs with the same maximum degree.
Moreover,
there are rarely results and methods for ordering in a giving graph set.
In 2006, Guo \cite{guo} determined the first $(\lfloor\frac{d}{2}\rfloor+2)$-th largest $Q$-indices (resp. Laplace spectral radius) of trees of giving order and diameter $d$. Very recently, Jia, Li and Wang \cite{Jia-Li} ordered the second to the
$(\lfloor\frac{d}{2}\rfloor+1)$-th $Q$-indices of graphs of giving size and diameter.

Inspired in above researches,  in this paper we intend to order of $Q$-indices in a giving graph set.
Girth is a graph variant that are
widely concerned by researchers in graph spectral theory.
 Chen, Wang and Zhai in 2022 gave the first largest $Q$-index of graph of giving size and girth.
Let $\mathbb{G}_m$ be the family of connected graphs  with  $m$ edges,
$\mathbb{G}(m,g)$ and  $\mathbb{G}(m,\geq g)$ are respectively the subset of $\mathbb{G}_m$  with girth equals to $g$ and  girth no less than $g$,
where $g\ge3$.
In this paper, we respectively consider  the order of the three families $\mathbb{G}(m,g)$,  $\mathbb{G}(m,\ge g)$ and $\mathbb{G}_m$   via their  $Q$-indices.

For $g\geq 3$, we always
denote by  $C_g=012\cdots(g-2)(g-1)0$  the   cycle of length $g$.
For $0\le i\le \lfloor\frac{g}{2}\rfloor$,
let $G_i\in \mathbb{G}(m, g)$ be a graph obtained from $C_g$ by attaching $m-g-1$ pendant edges to the vertex $0$ and simultaneously adding a pendant edge, say $iw$, at the vertex $i$  (see Fig.\ref{girth}).
Let $G_v$ be a graph obtained from $C_g$  by attaching  $m-g-2$ pendent edges  and a $P_3$,  respectively,  at the vertex $0$, where $P_3=0vv_1$ (see Fig.\ref{girth}).
The main results of this paper are presented as follows.

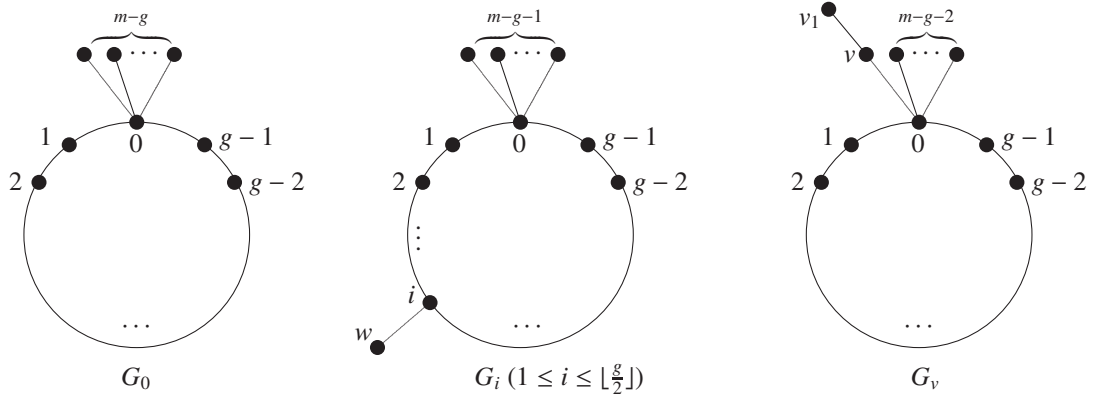
\begin{figure}[h]
  \centering
   \footnotesize
\unitlength 1mm 
\linethickness{0.4pt}
\ifx\plotpoint\undefined\newsavebox{\plotpoint}\fi 
\begin{picture}(136,51)(0,0)
\put(17,35){\circle*{2}}
\put(8,32){\circle*{2}}
\put(26,32){\circle*{2}}
\put(4,27){\circle*{2}}
\put(30,27){\circle*{2}}
\put(15,7){$\cdots$}
\multiput(10,44)(.033653846,-.043269231){208}{\line(0,-1){.043269231}}
\put(14,44){\line(1,-3){3}}
\multiput(22,44)(-.033557047,-.060402685){149}{\line(0,-1){.060402685}}
\put(10,44){\circle*{2}}
\put(14,44){\circle*{2}}
\put(22,44){\circle*{2}}
\put(16,44){$\ldots$}
\put(16,31){$0$}
\put(4,32){$1$}
\put(0,26){$2$}
\put(28,32){$g-1$}
\put(32,26){$g-2$}
\put(11,45){$\overbrace{ \qquad \ \ \ \ }^{m-g}$}
\put(68,35){\circle*{2}}
\put(59,32){\circle*{2}}
\put(77,32){\circle*{2}}
\put(55,27){\circle*{2}}
\put(81,27){\circle*{2}}
\put(54,18){$\vdots$}
\put(67,7){$\cdots$}
\multiput(61,44)(.033653846,-.043269231){208}{\line(0,-1){.043269231}}
\put(65,44){\line(1,-3){3}}
\multiput(73,44)(-.033557047,-.060402685){149}{\line(0,-1){.060402685}}
\put(61,44){\circle*{2}}
\put(65,44){\circle*{2}}
\put(73,44){\circle*{2}}
\put(67,44){$\ldots$}
\put(67,31){$0$}
\put(55,32){$1$}
\put(51,26){$2$}
\put(79,32){$g-1$}
\put(83,26){$g-2$}
\put(62,45){$\overbrace{ \qquad \ \ \ \ }^{m-g-1 }$}
\put(121,35){\circle*{2}}
\put(112,32){\circle*{2}}
\put(130,32){\circle*{2}}
\put(108,27){\circle*{2}}
\put(134,27){\circle*{2}}
\put(119,7){$\cdots$}
\multiput(114,44)(.033653846,-.043269231){208}{\line(0,-1){.043269231}}
\put(118,44){\line(1,-3){3}}
\multiput(126,44)(-.033557047,-.060402685){149}{\line(0,-1){.060402685}}
\put(114,44){\circle*{2}}
\put(118,44){\circle*{2}}
\put(126,44){\circle*{2}}
\put(120,44){$\ldots$}
\put(120,31){$0$}
\put(108,32){$1$}
\put(104,26){$2$}
\put(132,32){$g-1$}
\put(136,26){$g-2$}
\put(118,45){$\overbrace{}^{m-g-2 }$}
\put(56,11){\circle*{2}}
\multiput(56,11)(-.039325843,-.033707865){178}{\line(-1,0){.039325843}}
\put(49,5){\circle*{2}}
\put(53,11){$i$}
\put(46,6){$w$}
\put(15,0){$G_0$}
\put(62,0){$G_i$ ($1\le i\le \lfloor\frac{g}{2}\rfloor$)}
\put(120,0){$G_v$}
\put(114,44){\line(-5,6){5}}
\put(109,50){\circle*{2}}
\put(111,43){$v$}
\put(105,48){$v_1$}
\put(17,20){\circle{30}}
\put(68,20){\circle{30}}
\put(121,20){\circle{30}}
\end{picture}
  \caption{\footnotesize The graphs $G_0$, $G_i$  ($1\le i\le \lfloor\frac{g}{2}\rfloor$) and $G_v$}\label{girth}
\end{figure}

\begin{thm}\label{order-girth}
Among all graphs  in $\mathbb{G}(m, g)$ with $m\ge 3g\ge 12$,   the order of  the first $(\lfloor\frac{g}{2}\rfloor+2)$  largest  $Q$-indices of graphs
is given by:
$$q(G_0)> q(G_1)> q(G_v)> q(G_2)>q(G_3)> \cdots > q(G_{\lfloor\frac{g}{2}\rfloor}).$$
\end{thm}

We use $G_{i, g}$ and $G_{v, g}$ instead of  $G_i$ ($0\le i\le \lfloor\frac{g}{2}\rfloor$) and $G_v$,  respectively,  if we emphasize that its girth equals to $g$.
For the order of graphs in $\mathbb{G}(m, \geq g)$,
we get the  following result.
\begin{thm}\label{order-girth-1}
Among all graphs  in $\mathbb{G}(m, \geq g)$ with $m\ge 3g\ge 12$,   the order of the first $(\lfloor\frac{g}{2}\rfloor+3)$ largest $Q$-indices of graphs is given by:
$$q(G_{0,g})> q(G_{1,g})> q(G_{v,g})> q(G_{2,g})>q(G_{3,g})> \cdots > q(G_{\lfloor\frac{g}{2}\rfloor,g})>q(G_{0,g+1}).$$
\end{thm}

Let $B_1$ be the bicycle graph obtained from
two triangles with a common edge by join $m-5$ pendant edges to its one end-vertex (see Fig.\ref{graphsg=3}).
Let $B_2$ be the bicycle graph obtained from
two triangles with a common vertex by join $m-6$ pendant edges to it (see Fig.\ref{graphsg=3}).
As   supplement  of Theorem \ref{order-girth},  we give the order of the first five largest $Q$-indices for $g=3$,  which has some different with the result of
Theorem \ref{order-girth}.
\begin{thm}\label{order-g=3}
Among all graphs  in $\mathbb{G}(m, 3)$ with $m\ge 9$,   the order of the
 first five largest $Q$-indices  is given by:
 $q(G_{0,3})>q(B_1)>q(B_2)>q(G_{1,3})>q(G_{v,3}).$
\end{thm}

Zhai, Xue and Lou \cite{zhai-xue-lou} showed that  $K_{1,m}$ attains the maximum $Q$-index among all
graphs in $\mathbb{G}_m$.
At last, we extend  the result by ordering the first eleven largest $Q$-indices among all graphs in $\mathbb{G}_m$.

\begin{thm}\label{order-m}Let $T_1$, $T_2$, $T_3$ and $T_4$ be trees on $m$ edges with $\Delta(T_1)=m-1$ and  $\Delta(T_2)=\Delta(T_3)=\Delta(T_4)=m-2$ (see Fig.\ref{some-graph}).
Among all graphs in $\mathbb{G}_m$ and $m\ge 9$,
 the order of the first eleven largest $Q$-indices  is given by:
$q(K_{1,m})>q(G_{0,3})$$>q(T_1)>q(B_1)$$>q(B_2)$
 $>q(G_{1,3})>q(G_{v,3})
>q(T_2)>q(G_{0,4})>q(T_3)>q(T_4)$.
\end{thm}

\section{Some lemmas and a new upper bound of $q(G)$}
In the section, we give some useful lemmas   and then give a new bound of $q(G)$.

\begin{lem}[\cite{Cvetkovic1}]\label{lem-subgraph-radius}
If $H$ is the subgraph of  a connected graph $G$, then $q(H)\le q(G)$. Particularly, if $H$ is proper then $q(H)< q(G)$.
\end{lem}

\begin{lem}[\cite{Hong}]\label{lem-perron}
Let $u$, $v$ be two distinct vertices of a connected graph $G$. Suppose  $w_1,w_2,\ldots,w_t$ $(t\ge1)$ are some vertices of $N_G(v)\setminus N_G(u)$ and $\mathbf{x}$ is the   Perron vector of $Q(G)$. Let $G'=G-\{vw_i\mid i=1,2,\ldots,t\}+\{uw_i\mid i=1,2,\ldots,t\}$.  If $x_u\ge x_v$ then $q(G)<q(G')$.
\end{lem}

The following lemma gives an interesting transformation that could increase the $Q$-index.  As usual,  we call it as the `quadrangle' principle,
which is a useful tool in our proofs.
\begin{lem}[\cite{C2}]\label{lem-four-compare}
Let $G'$ be a graph obtained from a connected graph $G$ by a local switching of edges $ab$
and $cd$ to the positions of non-edges $ad$ and $bc$. Let $\mathbf{x}$ be the Perron vector of $Q(G)$. If $(x_a-x_c)(x_d-x_b)>0$ then $q(G')> q(G)$.
\end{lem}


%
%

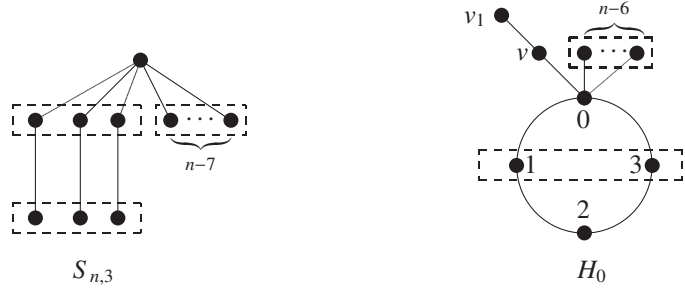
\begin{figure}
  \centering
  \footnotesize
\unitlength 1mm 
\linethickness{0.4pt}
\ifx\plotpoint\undefined\newsavebox{\plotpoint}\fi 
\begin{picture}(89,35.944)(0,0)
\put(76,23.944){\circle*{2}}
\put(76,29.944){\line(0,-1){6}}
\multiput(83,29.944)(-.039325843,-.033707865){178}{\line(-1,0){.039325843}}
\put(76,29.944){\circle*{2}}
\put(83,29.944){\circle*{2}}
\multiput(76,22.944)(-.0333333,.0333333){30}{\line(0,1){.0333333}}
\put(76,23.944){\line(-1,1){11}}
\put(70,29.944){\circle*{2}}
\put(65,34.944){\circle*{2}}
\put(78,30){$\ldots$}
\put(75,20){$0$}
\put(68,14){$1$}
\put(75,8){$2$}
\put(82,14){$3$}
\put(67,29){$v$}
\put(60,34){$v_1$}
\put(76,32){$\overbrace{}^{n-6}$}
\put(75,0){$H_0$}
\put(76,6){\circle*{2}}
\put(67,15){\circle*{2}}
\put(85,15){\circle*{2}}
\put(17,29){\circle*{2}}
\multiput(17,29)(-.058823529,-.033613445){238}{\line(-1,0){.058823529}}
\put(3,21){\circle*{2}}
\put(3,21){\line(0,-1){13}}
\put(3,8){\circle*{2}}
\put(17,29){\line(-1,-1){8}}
\put(9,21){\circle*{2}}
\put(9,21){\line(0,-1){13}}
\put(9,8){\circle*{2}}
\multiput(17,29)(-.03370787,-.08988764){89}{\line(0,-1){.08988764}}
\put(14,21){\circle*{2}}
\put(14,21){\line(0,-1){13}}
\put(14,8){\circle*{2}}
\put(17,29){\line(1,-2){4}}
\put(17,29){\line(3,-2){12}}
\put(21,21){\circle*{2}}
\put(29,21){\circle*{2}}
\put(23,21){$\ldots$}
\put(21,19){$\underbrace{}_{n-7}$}
\put(8,0){$S_{n,3}$}
\put(62,13){\dashbox{1}(27,4)[cc]{}}
\put(74,28){\dashbox{1}(11,4)[cc]{}}
\put(19,19){\dashbox{1}(12,4)[cc]{}}
\put(0,19){\dashbox{1}(17,4)[cc]{}}
\put(0,6){\dashbox{1}(17,4)[cc]{}}
\put(76,15){\circle{18}}
\end{picture}
\caption{\footnotesize The equitable partition of $S_{n,3}$ and $H_0$}\label{partition}
\end{figure}

The equitable partition is a significant tool for graph spectral theory.
The largest eigenvalue of the quotient matrix corresponding to an equitable partition of matrix $M$ is the largest eigenvalue of $M$ (see \cite{Brouwer}, Lemma 2.3.1).  Let $S_{n,3}$ be a graph on $n$ vertices obtained from $K_{1,n-7}$ by attaching  three pendant paths of length $2$ at the center vertex of   $K_{1,n-7}$, and $H_0$ be an unicycle graph of order $n$ and girth $4$ (see Fig. \ref{partition}).
It is routine to verify that $\phi_1(x,n)$ and $\phi_2(x,n)$ in Lemma \ref{Gv-subgraph} are respectively  the characteristic polynomial of the quotient matrix of $S_{n,3}$ and $H_0$, naturally we have the following result.
\begin{lem}\label{Gv-subgraph}
(i). The $Q$-index of   $S_{n,3}$ is the largest root of
$\phi_1(x,n)=x^3 - nx^2 + (3n-8)x - n$.\\
(ii). The $Q$-index of $H_0$ is the largest root of
$$\phi_2(x,n)=x^5 - ( n +5)x^4 + (7n + 1)x^3 - (15n-17)x^2 + (10n - 8)x - 2n.$$
\end{lem}

Let $\mathbf{x}$ be the Perron vector of $Q(G)$ with respect to $q(G)$ and $x_u$   the  coordinate of $\mathbf{x}$ corresponds to vertex $u$ ($u\in V(G)$).
By the  eigenvalue equation $Q(G)\mathbf{x}=q(G)\mathbf{x}$,
we have $q(G)x_u=d(u)x_u+\sum_{v\in N_G(u)}x_v$. Thus we can get a upper bound of $x_u$ for any $u\in V(G)$,  which can be used to  compare the coordinates of two vertices.
\begin{lem}\label{entry-upbound}
Let $G$ be a connected graph.
For any $u\in V(G)$,
 we have
$$x^2_u\le \frac{1}{1+\frac{(q(G)-d(u))^2}{d(u)}}.$$
\end{lem}
\begin{proof}
From the Cauchy-Schwarz inequality, we have
$$d(u)\sum_{v\in N_G(u)}x^2_v=\sum_{v\in N_G(u)}1^2\sum_{v\in N_G(u)}x^2_v\ge \left(\sum_{v\in N_G(u)}x_v\right)^2=\left( (q(G)-d(u))\cdot x_u \right)^2,$$
which implies
$$\sum_{v\in N_G(u)}x^2_v\ge\frac{\left((q(G)-d(u)\right)^2 x_u^2}{d(u)}.$$
Thus,
$$1=\sum_{i\in V(G)}x^2_i\ge x_u^2+\sum_{v\in N_G(u)}x^2_v\ge x_u^2+\frac{((q(G)-d(u))^2 x_u^2}{d(u)}=\left(1+\frac{((q(G)-d(u))^2}{d(u)}\right)x_u^2,$$
which follows
the result.
\end{proof}

\begin{lem}[\cite{Feng}]\label{bound}
Let $G$ be a connected graph. Then
$$q(G) \le \max_{u \in V (G)}\left\{d(u) +\frac{\sum_{v\in N_G(u)} d(v)}{d(u)} \right\},$$ with equality if and only if $G$ is either a semiregular bipartite graph or a regular graph.
\end{lem}
We now
 give an upper and lower bound of $Q$-index by Lemma \ref{bound},   from which we can  deduce three useful  corollaries for the later use.
\begin{thm}\label{outG-0}
For any connected $G$ with size $m\ge5$,  we have \\
(i) if $\Delta(G )\leq s$ and  $s\geq \frac{2m}{3}$,  then  $q(G)\le s+2$.\\
(ii) if  $s\leq\Delta(G ) \leq m-1$, then $q(G)> s+1$.
\end{thm}
\begin{proof}
Let $z\in V(G)$  such that
$$d(z) +\frac{\sum_{v\in N_G(z)} d(v)}{d(z)}=
\max_{u \in V (G)}\left\{d(u) +\frac{\sum_{v\in N_G(u)} d(v)}{d(u)}\right\}.$$
If  $d(z)=1$,  by Lemma \ref{bound},  we have
$q(G)\leq d(z) +\frac{\sum_{v\in N_G(z)} d(v)}{d(z)}\le1+\Delta(G) \le s+1.$
It follows the result.
If  $d(z)=2$,  by Lemma \ref{bound} we get
$q(G)\leq d(z) +\frac{\sum_{v\in N_G(z)} d(v)}{d(z)}\le 2+\Delta(G) \le s+2.$
It follows the result.
Next we consider the case of $d(z)\geq 3$.
Note that
\begin{eqnarray}\label{u-bound}
q(G)\leq d(z) +\frac{\sum_{v\in N_G(z)} d(v)}{d(z)}\le
d(z) +\frac{2m-d(z)}{d(z)}=d(z) +\frac{2m}{d(z)}-1.\end{eqnarray}
Clearly, $3\le d(z)\le \Delta(G)\le s$.
Let $f(x)=x+\frac{2m}{x}$.
Then $f(x)$ is decreased in the internal $[3,\sqrt{2m}]$ and increased in the internal $[\sqrt{2m}, +\infty)$.
Since   $s\ge \frac{2m}{3}$,  also noticed that $f(3)=f(\frac{2m}{3})$ and
$\frac{2m}{3}>\sqrt{2m}$ due to $m\ge5$,
we have
\begin{eqnarray}\label{u-function}
d(z) +\frac{2m}{d(z)}-1\le s+\frac{2m}{s}-1\le s+\frac{2m}{\frac{2m}{3}}-1=s+2.
\end{eqnarray}
Combining  (\ref{u-bound}) and (\ref{u-function}), we have the first result.

If  $s\leq\Delta(G ) \le m-1$, then $G$ has a $K_{1, s}$ as a proper subgraph, and so $q(G)>q(K_{1, s})=s+1$ from Lemma \ref{lem-subgraph-radius}.
It follows the results.
\end{proof}

Recall that   $\mathbb{G}(m,\geq g)$ is the set of connected graphs on $m$ edges and girth  no less than $g$, where $g\ge3$, and $G_0\in \mathbb{G}(m, g)$ is a graph obtained from $C_g$ by attaching $m-g$ pendant edges to  $0$.
By simple observation,   we see that $G_0$ is the unique graph among  $\mathbb{G}(m, \geq g)$ with maximum degree $\Delta(G_0)=m-g+2$.
Moreover, we have the following result.
\begin{cor}\label{max-graph-girth}
Let $G\in \mathbb{G}(m, \geq g)$ with $m\ge 3g-3$. Then $q(G)\le q(G_0)$, with equality if and only if $G\cong G_0$.
\end{cor}
\begin{proof}
By the definition of $\mathbb{G}(m, \geq g)$,
taking any $G\in  \mathbb{G}(m, \geq g)\setminus \{G_0\}$, we have $\Delta(G) \le m-g+1$.
Note that $m-g+1\geq \frac{2m}{3}$ since $m\ge 3g-3$. By Theorem \ref{outG-0} (i), we have $q(G)\leq m-g+3$.
On the other hand, we know that $m-g+2=\Delta(G_0) \le m-1$ since $g\ge 3$,  By Theorem \ref{outG-0}(ii), we have  $q(G_0)>m-g+3\geq q(G)$.
\end{proof}
Recently, by using different ways Chen, Wang and Zhai in \cite{Zhai} have obtained  the  result of Corollary \ref{max-graph-girth}.
By Theorem \ref{outG-0},  we will give a relation of $Q$-indices of graphs between two distinct girths.
\begin{cor}
Let $G^*$ and $H^*$ respectively be  graph with the maximum $Q$-index in $\mathbb{G}(m, g)$ and $\mathbb{G}(m, g')$.
If $g<g'$ and $m\geq 3g'-3$,
then $q(G^*)>q(H^*)$.
\end{cor}
\begin{proof}Since $m\geq 3g'-3$ and $g'>g$, we have $m\geq 3g-3$.
By Corollary \ref{max-graph-girth},  we get that $G^*$ (resp. $H^*$) is isomorphic to an  unicycle graph $C_g$ (resp. $C_{g'}$)
by attaching  $m-g$ (resp. $m-g'$)
pendent edges to the same vertex of the cycle.
Clearly,
$\Delta(G^*)=m-g+2$ and $\Delta(H^*)=m-g'+2$.
Notice that  $\Delta(H^*)=m-g'+2\geq \frac{2m}{3}$ since $m\geq 3g'-6$.
We have $q(H^*)\leq m-g'+4$ from Theorem \ref{outG-0} (i).
Note that
 $\Delta(G^*)=m-g+2\leq m-1$.
By Theorem \ref{outG-0} (ii), we have $q(G^*)>m-g+3\geq m-g'+4 \geq q(H^*)$. It follows the result.
\end{proof}

By Theorem \ref{outG-0}, we can get a relation of $Q$-indices of graphs between two maximum degrees.
\begin{cor}\label{compare-max-degree}
Let $G$ and $H$ be graphs with size $m\ge5$ and   maximum degree $\Delta(G)$ and $\Delta(H)$, respectively.
If $m-1\ge \Delta(G)>\Delta(H) \geq \frac{2m}{3}$,
then $q(G)>q(H)$.
\end{cor}

\begin{proof}
Since $\Delta(G) \le m-1$,  we have $q(G)>\Delta(G)+1\geq \Delta(H)+2$  by Theorem \ref{outG-0}(ii).
On the other hand,  note that $\Delta(H)\geq \frac{2m}{3}$.
By Theorem \ref{outG-0}(i), we have  $q(H)\leq \Delta(H)+2< q(G)$. It follows the result.
\end{proof}

\section{The order of  $Q$-indices of graphs in $\mathbb{G}(m,g)$}
In the section, we will give the order of graphs in $\mathbb{G}(m,g)$ via  their $Q$-indices.
For any $G\in \mathbb{G}(m, g)$,
$C_g=012\cdots(g-2)(g-1)0$ is always denoted by one of a shortest  cycle of $G$.
If $m=g$ then $\mathbb{G}(m,g)=\{C_g\}$. If $m=g+1$ then $\mathbb{G}(m,g)=\{C^+_g\}$, where $C^+_g$ is a graph obtained from $C_g$ by attaching a pendant edge at some vertex of $C_g$.
In what follows, we consider $m\ge g+2$ and the corresponding $|\mathbb{G}(m,g)|\geq 2$.

For $g\ge3$ and $m\ge g+2$,
let $\mathbb{G}_{\Delta}(m, g)$ be the set of graphs in $\mathbb{G}(m, g)$ with maximum degree  $\Delta=m-g+1$.
Recall that
 $G_i\in \mathbb{G}(m, g)$ is  obtained from $C_g$ by attaching $m-g-1$ pendant edges at  vertex $0$ and simultaneously adding a pendant edge at the vertex $i$,  say $iw$,  for $1\le i\le \lfloor\frac{g}{2}\rfloor$ (see Fig.\ref{girth}), and
 $G_v$ is obtained from $C_g$  by attaching, respectively, $m-g-2$  pendent edges  and a $P_3$ at  vertex $0$, where $P_3=0vv_1$ (see Fig.\ref{girth}).

\begin{lem}\label{degree-1}
 $\mathbb{G}_{\Delta}(m, g)=\{G_1, G_2, \ldots, G_{\lfloor\frac{g}{2}\rfloor}, G_v\}$,  where $g\ge4$ and $m\ge g+2$.
\end{lem}
\begin{proof}
Let $G\in\mathbb{G}_{\Delta}(m, g)$ with a cycle $C_g=012\cdots(g-2)(g-1)0$.  Without loss of generality, we may assume that $d(0)=\Delta(G)=m-g+1\ge 3$.
Let $G'$ be a subgraph of $G$ induced by
$V(C_g)\cup N_{G}(0)$.
 We have
$m(G')=|C_g|+d(0)-2=m-1$.   Thus $G$ can be obtained from  $G'$ by adding an edge $e$.
Denote by  $v\not\in C_g$  a vertex adjacent with $0$.
By the minimality of  the length of $C_g$ and  $g\ge4$, we get that
$e$ must be a pendent edge attaching  one  vertex  of   $\{1,2,\ldots,g-1,v\}$. By considering  the symmetry of  the vertices $i$ and $g-i$ in $C_g$, we have
 $\mathbb{G}_{\Delta}(m, g)=\{G_1, G_2, \ldots,G_{\lfloor\frac{g}{2}\rfloor}, G_v\}$.
\end{proof}

\begin{lem}\label{Gi-compare}
Let  $G_i\in \mathbb{G}_{\Delta}(m, g)$ shown in Fig.\ref{girth}, where $1\le i\le \lfloor\frac{g}{2}\rfloor$. Then $q(G_1)>q(G_{2})>\cdots>q(G_{\frac{g}{2}})$.
\end{lem}
\begin{proof}
For any  $2\le i \le\lfloor\frac{g}{2}\rfloor$, let $\mathbf{x}=(x_u)$ be the Perron vector of $Q(G_i)$, where  $u\in V(G_i)$.
It suffices   to show $q(G_i)<q(G_{i-1})$. To prove our  result, first we give the following claim.
{\flushleft\bf Claim 1.} If there exists $1\le j \le i-1$ such that $x_{i-j}<x_{i+j-1}$ and $x_{i-j-1}\ge x_{i+j}$, then $q(G_i)<q(G_{i-1})$.
\begin{proof}
In fact, let $G'=G_i-\{(i-j-1)( i-j), (i+j-1)(i+j)\}+\{(i-j-1)( i+j-1), (i-j)(i+j)\}$ (see Fig.\ref{case2-girth}).  One can observe that $G'\cong G_{i-1}$.
By Lemma \ref{lem-four-compare}, we have $q(G_i)<q(G')=q(G_{i-1})$.\end{proof}

We start to prove by firstly assuming $x_{i-1}\ge x_i$. Now we construct  $G''=G_i-wi+w(i-1)$ from $G_i$. It is clear that $G''\cong G_{i-1}$.
By Lemma \ref{lem-perron}, we have $q(G_i)<q(G'')=q(G_{i-1})$. Otherwise $x_{i-1}< x_i$, if $x_{i-2}\ge x_{i+1}$ then from Claim 1 we  get $q(G_i)<q(G_{i-1})$ by taking $j=1$. Otherwise  $x_{i-2}< x_{i+1}$, if $x_{i-3}\ge x_{i+2}$ then from Claim 1 we  get $q(G_i)<q(G_{i-1})$ by taking $j=2$. Repeating $i$ steps we come to the assumption $x_{0}< x_{2i-1}$ for $j=i$.
Let $N_{G_i}(0)=\{1,g-1, w_1,\ldots, w_{m-g-1}\}$ and $G'''=G_i-\{0w_t \mid 1\le t\le m-g-1\}+\{(2i-1)w_t \mid 1\le t\le m-g-1\}$. Clearly, $G'''\cong G_{i-1}$.
By Lemma \ref{lem-perron}, we have $q(G_i)<q(G''')=q(G_{i-1})$.

It completes the proof.
\end{proof}

\begin{figure}[h]
  \centering
   \footnotesize
\unitlength 1mm 
\linethickness{0.4pt}
\ifx\plotpoint\undefined\newsavebox{\plotpoint}\fi 
\begin{picture}(121,48.5)(0,0)
\put(27,38.5){\circle*{2}}
\put(18,35.5){\circle*{2}}
\put(36,35.5){\circle*{2}}
\put(27,10.5){$\cdots$}
\multiput(20,47.5)(.033653846,-.043269231){208}{\line(0,-1){.043269231}}
\put(24,47.5){\line(1,-3){3}}
\multiput(32,47.5)(-.033557047,-.060402685){149}{\line(0,-1){.060402685}}
\put(20,47.5){\circle*{2}}
\put(24,47.5){\circle*{2}}
\put(32,47.5){\circle*{2}}
\put(26,47.5){$\ldots$}
\put(26,34.5){$0$}
\put(14,35.5){$1$}
\put(38,35.5){$g\!-\!1$}
\put(21,48.5){$\overbrace{ \qquad \ \ \ \ }^{m-g-1}$}
\put(18,11.5){\circle*{2}}
\multiput(18,11.5)(-.039325843,-.033707865){178}{\line(-1,0){.039325843}}
\put(17,7.5){$i$}
\put(36,11.5){\circle*{2}}
\put(41,17.5){\circle*{2}}
\put(13,29.5){\circle*{2}}
\put(11,5.5){\circle*{2}}
\put(10,2.5){$w$}
\put(12,21.5){\circle*{2}}
\put(4,21){$i\!-\!j$}
\put(0,29){$i\!-\!j\!-\!1$}
\put(38,10){$i\!+\!j\!-\!1$}
\put(43,17){$i\!+\!j$}
\put(40,25){$\vdots$}
\put(14,15){$\ddots$}
\put(24,0){$G_i$}
\put(105,38.5){\circle*{2}}
\put(96,35.5){\circle*{2}}
\put(114,35.5){\circle*{2}}
\put(105,10.5){$\cdots$}
\multiput(98,47.5)(.033653846,-.043269231){208}{\line(0,-1){.043269231}}
\put(102,47.5){\line(1,-3){3}}
\multiput(110,47.5)(-.033557047,-.060402685){149}{\line(0,-1){.060402685}}
\put(98,47.5){\circle*{2}}
\put(102,47.5){\circle*{2}}
\put(110,47.5){\circle*{2}}
\put(104,47.5){$\ldots$}
\put(104,34.5){$0$}
\put(92,35.5){$1$}
\put(116,36.5){$g\!-\!1$}
\put(99,48.5){$\overbrace{ \qquad \ \ \ \ }^{m-g-1}$}
\put(96,11.5){\circle*{2}}
\multiput(96,11.5)(-.039325843,-.033707865){178}{\line(-1,0){.039325843}}
\put(95,7.5){$i$}
\put(114,11.5){\circle*{2}}
\put(119,17.5){\circle*{2}}
\put(91,29.5){\circle*{2}}
\put(89,5.5){\circle*{2}}
\put(88,2.5){$w$}
\put(82,20.5){$i\!-\!j$}
\put(79,29.5){$i\!-\!j\!-\!1$}
\put(115,8.5){$i\!+\!j\!-\!1$}
\put(121,15.5){$i\!+\!j$}
\put(115,5.5){$(i\!-\!j)$}
\put(93,4.5){$(i\!-\!1)$}
\put(78,16.5){$(i\!+\!j\!-\!1)$}
\put(116,25){$\vdots$}
\put(93,15){$\ddots$}
\put(88.5,24){\XSolidBold}
\put(115,12){\XSolidBold}
\put(104,0){$G'$}
\multiput(91,29)(.0436507937,-.0337301587){504}{\line(1,0){.0436507937}}
\put(90,22){\circle*{2}}
\multiput(90,22)(.235294118,-.033613445){119}{\line(1,0){.235294118}}
\put(27,23.5){\circle{30}}
\put(105,23.5){\circle{30}}
\end{picture}

  \caption{\footnotesize $G_i$ and $G'$ used in the proof of  Lemma \ref{Gi-compare},  where the edge with ``\XSolidBold'' represents it is deleted.}\label{case2-girth}
\end{figure}
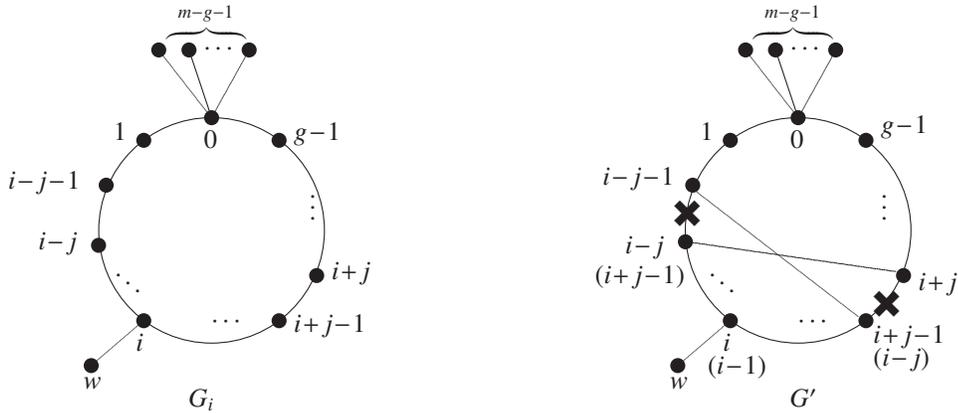

Next we will give a lower and upper bound for $q(G_i)$.
\begin{lem}\label{G_i-bound}
Let  $G_i\in \mathbb{G}_{\Delta}(m, g)$ with  $1\le i\le \lfloor\frac{g}{2}\rfloor$ shown in  Fig.\ref{girth}. If   $m\ge g+3$,
we have $$m-g+2 < q(G_i) \le m-g+2+\frac{2}{ m-g+1}<m-g+3.$$
\end{lem}
\begin{proof}
Since $\Delta(G_i)=m-g+1\le m-1$, by Theorem \ref{outG-0}(ii) we have $q(G_i)> m-g+2$. It suffices to show $q(G_i) \le m-g+2+\frac{2}{ m-g+1}$.
We consider $d(u) +\frac{\sum_{v\in N_{G_i}(u)} d(v)}{d(u)}$ by distinguishing the following situations.

If $u=0$, then
$d(u) +\frac{\sum_{v\in N_{G_i}(u)} d(v)}{d(u)}= m-g+2+\frac{2}{ m-g+1}$.
If $u=i$, we have
$$\begin{aligned}
d(u) +\frac{\sum_{v\in N_{G_i}(u)} d(v)}{d(u)}\leq &d(i)+\frac{d(w)+d(0)+d(i+1)}{d(i)}\leq 3+\frac{1+(m-g+1)+2}{3}\\
<&m-g+2+\frac{2}{ m-g+1} \ \mbox{( because $m\ge g+3$ )}.
\end{aligned}$$
If $u$ is a pendent vertex of $G_i$, then
$d(u) +\frac{\sum_{v\in N_{G_i}(u)} d(v)}{d(u)}\leq 1+d(0)=m-g+2< m-g+2+\frac{2}{ m-g+1}.$
If $u\in V(G_i)\setminus \{0,i\}$ is not a pendent vertex,
we have
$d(u) +\frac{\sum_{v\in N_{G_i}(u)} d(v)}{d(u)}\leq2+\frac{d(0)+d(i)}{2}\leq2+\frac{m-g+4}{2}\leq m-g+2+\frac{2}{ m-g+1}$ for $m\ge g+3$.

Thus by Lemma \ref{bound}, we have
$$q(G_i)\leq\max_{u \in V (G_i)}\left\{d(u) +\frac{\sum_{v\in N_{G_i}(u)} d(v)}{d(u)}\right\}= m-g+2+\frac{2}{ m-g+1}.$$
Thus the result holds.
\end{proof}

\begin{lem}\label{max-entry}
For $1\le i\le \lfloor\frac{g}{2}\rfloor$ and $m\ge\max\{2g-2, g+7\}$,
let $\mathbf{x}$ be  the Perron vector of $Q(G_i)$. Then  $x_0$ is the maximum entry of  $\mathbf{x}$.
\end{lem}
\begin{proof}
Let $q=q(G_i)$ and $v$, $w$ be a pendant vertex attaching to the vertex $0$ and $i$ of $G_i$, respectively (see Fig.\ref{girth}).
By eigenvalue equation, we have $x_v=\frac{1}{q-1}x_0$ and $x_w=\frac{1}{q-1}x_i$.
Note that $d(i)=3$ and $d(j)=2$ for $1\le j\le g-1$ and $j\not=i$.
By Lemma \ref{entry-upbound}, we have
 $$x_i^2\le \frac{1}{1+\frac{(q-3)^2}{3}}    \ \ \mbox{and} \ \
 x_j^2\le \frac{1}{1+\frac{(q-2)^2}{2}}.$$
Furthermore, we get that
$$
 \begin{aligned}
 1=&\sum_{u\in V(G_i)}x^2_u
 =x_0^2+(m-g-1)x_v^2+\sum_{j=1,j\not=i}^{g-1}x_j^2+x^2_i+x^2_w\\
 \le& x_0^2+ (m-g-1)\cdot(\frac{1}{q-1}x_0)^2+(g-2)\cdot\frac{1}{1+\frac{(q-2)^2}{2}}+(1+\frac{1}{(q-1)^2})\cdot\frac{1}{1+\frac{(q-3)^2}{3}}
\end{aligned}
$$
It follows that
\begin{eqnarray*}
x^2_0\ge \frac{1-(g-2)\cdot\frac{1}{1+\frac{(q-2)^2}{2}}-(1+\frac{1}{(q-1)^2})\cdot\frac{1}{1+\frac{(q-3)^2}{3}}}{1+\frac{m-g-1}{(q-1)^2}},
\end{eqnarray*}
which is equavalent to
\begin{eqnarray}\label{x_0-girth}
x^2_0\ge \frac{h_1(q)+h_2(q)}{2h_3(q)}+\frac{1}{2},
\end{eqnarray}
where
$$\left\{\begin{array}{ll}
h_1(q)=q^6 - 17q^5 + 110q^4 - 378q^3 + 716q^2 - 720q + 312,\\
h_2(q)=(q^4 - 10q^3 + 42q^2 - 84q + 72)(q-(m-g+2))+
4(q^2 - 6q + 12)( q - 1)^2(q-g),\\
h_3(q)=( q^2 - 4q + 6)( q^2 - 6q + 12)( q^2 - 2q + m-g).
\end{array}\right.$$
Recall that $\Delta(G_i)=m-g+1\le m-1$, by Theorem \ref{outG-0}(ii)
 we have $q>m-g+2\ge9$ due to $m\ge g+ 7$.   Moreover, notice that the polynomials below are  constant coefficients,  using the computer
we  get  for $q>9$ that
$$h_1(q)>0, q^4 - 10q^3 + 42q^2 - 84q + 72>0,  q^2 - 6q + 12>0,  q^2 - 4q + 6>0. $$
Since $m\ge2g-2$, we have $q>m-g+2\ge g$ and so $q-(m-g+2), q-g>0$. Thus $h_2(q)> 0$.
Clearly, $m-g\ge0$. Then we have $q^2 - 2q + m-g>0$, and so $h_3(q)>0$.
From (\ref{x_0-girth}), we have $x_0^2> \frac{1}{2}$ and thus $x^2_u<\frac{1}{2}$ for  any $u\in V(G_i)\backslash\{0\}$.

It completes the proof.
\end{proof}

\begin{lem}\label{three-compare}
If $g\ge4$ and $m\ge\max\{2g-2, g+7\}$, then $q(G_1)>q(G_{v})>q(G_2)$.
\end{lem}
\begin{proof}
We first prove $q(G_1)>q(G_{v})$.
Let  $\mathbf{x}$ be the Perron vector of $Q(G_v)$.
The vertices $v$ and $v_1$ of  $G_v$ are shown in Fig.\ref{girth}.
By the eigenvalue  equation, we have
\begin{eqnarray}\label{G-v-eq}
 \begin{aligned}
q(G_v)x_{v_1}=&x_{v_1}+x_v, \  \
q(G_v)x_v=2x_v+x_{v_1}+x_0,\\
q(G_v)x_0=&(m-g+1)x_0+\frac{m-g-2}{q(G_v)-1}x_0+x_v+x_1+x_{g-1}.
\end{aligned}
\end{eqnarray}
Note that $x_1=x_{g-1}$ due to the symmetry of $G_v$. From (\ref{G-v-eq}), we have
$$\left\{\begin{aligned}
x_v=&\frac{q(G_v)-1}{(q(G_v)-2)(q(G_v)-1)-1}x_0,\\
x_1=&\frac{1}{2}(q(G_v)-(m-g+1)-\frac{m-g-2}{q(G_v)-1}-\frac{q(G_v)-1}{(q(G_v)-2)(q(G_v)-1)-1})x_0.
\end{aligned}\right.
$$
Let
$$f(x)=\frac{1}{2}(x-(m-g+1)-\frac{m-g-2}{x-1}-\frac{x-1}{(x-2)(x-1)-1})-\frac{x-1}{(x-2)(x-1)-1}.$$
Then $x_1-x_v=f(q(G_v))x_0.$
Notice that
\begin{eqnarray}\label{cha-function}
f(x)=\frac{x}{2(x^2-3x+1)(x-1)}\phi_1(x,m-g+5),
\end{eqnarray}
where $\phi_1(x,m-g+5)$ is defined by Lemma \ref{Gv-subgraph}(i).
If $g\ge5$,
  then $S_{m-g+5,3}$ is a proper subgraph of $G_v$, and so $q(G_v)>q(S_{m-g+5,3})$.
Recall that $\Delta(G_v)=m-g+1\le m-1$, by Theorem \ref{outG-0}(ii)
 we have $q(G_v)>m-g+2\ge9$ due to $m\ge g+ 7$.
It is easy to versify that $x^2-3x+1>0$ for $x>9$.
Thus   $f(q(G_v))>0$, which implies that $x_1>x_v$.
If $g=4$,
 then (\ref{cha-function}) becomes
 $$f(x)=\frac{x}{2(x^2-3x+1)(x-1)}\cdot
 \frac{\phi_2(x,m)+2x-2}{x^2-4x+2},$$
where $\phi_2(x,m)$ is is defined by Lemma \ref{Gv-subgraph}(ii).
Clearly, $\phi_2(q(G_v),m)=0$.
On the other hand, we have $2x-2>0$ and $x^2-4x+2>0$ for $x>9$.
Recall that $q(G_v)>9$.
Thus   $f(q(G_v))>0$, and also $x_1>x_v$.
Let $G'=G_v-\{vv_1\}+\{1v_1\}$. Clearly, $G'\cong G_1$.
By Lemma \ref{lem-perron}, we have $q(G_v)<q(G')=q(G_1)$.

Next we will prove $q(G_{v})>q(G_2)$.
Let  $\mathbf{y}$ be the Perron vector of $Q(G_2)$.
From Lemma \ref{max-entry},
$y_0=\max_{u\in V(G_2)}\{y_u\}$.
Taking any vertex $u$ with degree $2$ in $G_2$,
by the eigenvalue equation, we have
$q(G_2)y_u=2y_u+\sum_{v\in N_{G_2}(u)} y_v\le 2y_u+2y_0,$
which implies that
\begin{eqnarray}\label{y-2-degree-bound}
y_u\le \frac{2}{q(G_2)-2}y_0.
\end{eqnarray}
By the eigenvalue  equation again,
$q(G_2)y_w=y_w+y_2$ and
$q(G_2)y_2=3y_2+y_w+y_1+y_3$.
Thus  from (\ref{y-2-degree-bound}) we have
$$y_2=\frac{ q(G_2)-1}{q^2(G_2)-4q(G_2)+2}(y_1+y_3)\le \frac{ q(G_2)-1}{q^2(G_2)-4q(G_2)+2}\cdot\frac{4}{q(G_2)-2}y_0.$$
Let $v$ be a pendent vertex attaching $0$ in $G_2$.  Clearly, $y_v=\frac{1}{q(G_2)-1}y_0$.
On the other hand,    one can easily verify that
$$\frac{ q(G_2)-1}{q^2(G_2)-4q(G_2)+2}\cdot\frac{4}{q(G_2)-2}<\frac{1}{q(G_2)-1}$$
for $q(G_2)\ge 8$.
Since $m\ge g +7$,  by Lemma \ref{G_i-bound}, we have $q(G_2)> m-g+2\ge 8$.
Hence, $y_2<y_v$.
Let $G''=G_2-\{2w\}+\{vw\}$.
 Clearly, $G''\cong G_v$.
By Lemma \ref{lem-perron}, we have $q(G_2)<q(G'')=q(G_v)$.

It completes the proof.
\end{proof}

It is time to provide the proofs of our main results.  First we prove Theorem \ref{order-girth} that orders the first  $(\lfloor\frac{g}{2}\rfloor+2)$ largest  graphs
according  their $Q$-indices among $\mathbb{G}(m, g)$.
\begin{proof}[\bf{Proof of Theorem \ref{order-girth}}]
Note that $G_0$ is a unique graph  with maximum degree $m-g+2$ among $\mathbb{G}(m, g)$
and  $m\ge 3g\ge 12$.
By Lemma \ref{degree-1}
$\mathbb{G}_{\Delta}(m, g)=\{G_1, G_2, \ldots, G_{\lfloor\frac{g}{2}\rfloor}, G_v\}$ is the set of graphs with maximum degree  $m-g+1$.
Since  $m-1\geq\Delta(G_0)=m-g+2>m-g+1=\Delta(G_1)\geq \frac{2m}{3}$,  by Corollary \ref{compare-max-degree}
we have $q(G_0)>q(G_1)$.
Moreover, notice that $g\geq 4$ and $m\geq \max\{2g-2, g+7\}$ due to $m\ge 3g\ge 12$,
by Lemmas \ref{Gi-compare} and \ref{three-compare}, we have
$q(G_0)>q(G_1)> q(G_v)> q(G_2)>(G_3)> \cdots > q(G_{\lfloor\frac{g}{2}\rfloor})$.
Set $\mathbb{G}_{m-g}(m, g)=\mathbb{G}(m, g)\setminus (\mathbb{G}_{\Delta}(m, g)\cup \{G_0\})$.
Then for any $G'\in \mathbb{G}_{m-g}(m, g)$,
we have $\Delta(G')\le m-g$.
Since $m\ge 3g$, we have $m-1\ge\Delta(G_{\lfloor\frac{g}{2}\rfloor}) =m-g+1>m-g=\Delta(G')\ge \frac{2m}{3}$.
By Corollary \ref{compare-max-degree}, we get
$q(G_{\lfloor\frac{g}{2}\rfloor})>q(G')$.
Thus the first $(\lfloor\frac{g}{2}\rfloor+2)$ largest $Q$-indices of all graphs among $\mathbb{G}(m, g)$
belong to $\mathbb{G}_{\Delta}(m, g)\cup \{G_0\}$, which is given by
\begin{eqnarray*}\label{one-degree}
q(G_0)>q(G_1)> q(G_v)> q(G_2)>(G_3)> \cdots > q(G_{\lfloor\frac{g}{2}\rfloor}).
\end{eqnarray*}
It completes the proof.
\end{proof}

 Next we prove Theorem \ref{order-girth-1} that orders the first $(\lfloor\frac{g}{2}\rfloor+3)$ largest graphs according their   $Q$-indices among $\mathbb{G}(m, \geq g)$.
For $0\le i\le \lfloor\frac{g}{2}\rfloor$,
we use $G_{i,g}$ and $G_{v,g}$ instead of $G_i$  and $G_v$ to distinguish the girth of the graphs in the following proofs.
\begin{figure}
  \centering
   \footnotesize
\unitlength 1mm 
\linethickness{0.4pt}
\ifx\plotpoint\undefined\newsavebox{\plotpoint}\fi 
\begin{picture}(135,51)(0,0)
\put(17,41){\circle*{2}}
\put(8,38){\circle*{2}}
\put(26,38){\circle*{2}}
\put(4,33){\circle*{2}}
\put(30,33){\circle*{2}}
\multiput(10,50)(.033653846,-.043269231){208}{\line(0,-1){.043269231}}
\put(14,50){\line(1,-3){3}}
\multiput(22,50)(-.033557047,-.060402685){149}{\line(0,-1){.060402685}}
\put(10,50){\circle*{2}}
\put(14,50){\circle*{2}}
\put(22,50){\circle*{2}}
\put(16,50){$\ldots$}
\put(16,37){$0$}
\put(4,38){$1$}
\put(0,32){$2$}
\put(28,38){$g$}
\put(32,32){$g-1$}
\put(28,12){$\lfloor\frac{g}{2}\rfloor\!+\!2$}
\put(13,7){$\lfloor\frac{g}{2}\rfloor\!+\!1$}
\put(2,12){$\lfloor\frac{g}{2}\rfloor$}
\put(7,0){$G_{0,g+1}$ for odd $g$}
\put(11,51){$\overbrace{ \qquad \ \ \ \ }^{m-g-1}$}
\put(68,41){\circle*{2}}
\put(59,38){\circle*{2}}
\put(77,38){\circle*{2}}
\put(55,33){\circle*{2}}
\put(81,33){\circle*{2}}
\multiput(61,50)(.033653846,-.043269231){208}{\line(0,-1){.043269231}}
\put(65,50){\line(1,-3){3}}
\multiput(73,50)(-.033557047,-.060402685){149}{\line(0,-1){.060402685}}
\put(61,50){\circle*{2}}
\put(65,50){\circle*{2}}
\put(73,50){\circle*{2}}
\put(67,50){$\ldots$}
\put(67,37){$0$}
\put(55,38){$1$}
\put(51,32){$2$}
\put(79,38){$g$}
\put(83,32){$g-1$}
\put(81,15){$\lfloor\frac{g}{2}\rfloor\!+\!2$}
\put(68,8){$\lfloor\frac{g}{2}\rfloor\!+\!1$}
\put(55,11){$\lfloor\frac{g}{2}\rfloor$}
\put(57,0){$G_{0,g+1}$ for even $g$}
\put(62,51){$\overbrace{ \qquad \ \ \ \ }^{m-g-1 }$}
\put(17,11){\circle*{2}}
\put(8,14){\circle*{2}}
\put(26,14){\circle*{2}}
\put(62,12){\circle*{2}}
\put(73,12){\circle*{2}}
\put(120,41){\circle*{2}}
\put(111,38){\circle*{2}}
\put(129,38){\circle*{2}}
\put(107,33){\circle*{2}}
\put(133,33){\circle*{2}}
\multiput(113,50)(.033653846,-.043269231){208}{\line(0,-1){.043269231}}
\put(117,50){\line(1,-3){3}}
\multiput(125,50)(-.033557047,-.060402685){149}{\line(0,-1){.060402685}}
\put(113,50){\circle*{2}}
\put(117,50){\circle*{2}}
\put(125,50){\circle*{2}}
\put(119,50){$\ldots$}
\put(119,37){$0$}
\put(107,38){$1$}
\put(103,32){$2$}
\put(131,38){$g$}
\put(135,32){$g-1$}
\put(133,15){$\lfloor\frac{g}{2}\rfloor\!+\!2$}
\put(120,8){$\lfloor\frac{g}{2}\rfloor\!+\!1$}
\put(107,11){$\lfloor\frac{g}{2}\rfloor$}
\put(117,0){$G'$}
\put(127,11.5){\XSolidBold}
\put(114,51){$\overbrace{ \qquad \ \ \ \ }^{m-g-1 }$}
\put(114,12){\circle*{2}}
\put(125,12){\circle*{2}}
\put(80,17){\circle*{2}}
\put(132,17){\circle*{2}}
\multiput(114,12)(.120805369,.033557047){149}{\line(1,0){.120805369}}
\put(17,26){\circle{30}}
\put(68,26){\circle{30}}
\put(120,26){\circle{30}}

\end{picture}

  \caption{ \footnotesize $G_{0,g+1}$ and $G'$ used in the Proof of Theorem \ref{order-girth-1},
 where the edge with ``\XSolidBold'' represents it is deleted.
  }\label{g+1}
\end{figure}
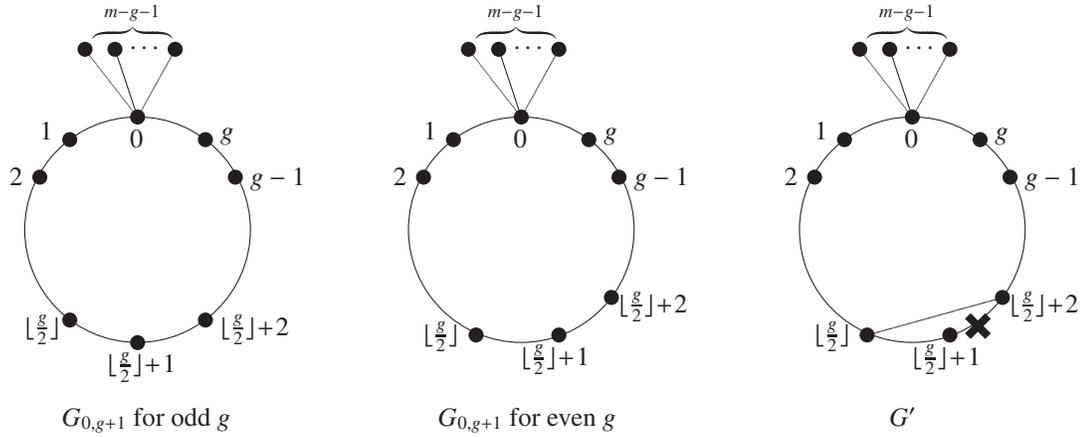

\begin{proof}[\bf{Proof of Theorem \ref{order-girth-1}}]
Denote by $\overline{\mathbb{G}}$ and $\underline{\mathbb{G}}$ respectively the sets of all graphs in $\mathbb{G}(m, \geq g)$  with maximum degree no less than $m-g+1$ and no more than $m-g$.
Then we have a partition  $\mathbb{G}(m, \geq g)=
\overline{\mathbb{G}}
\cup \underline{\mathbb{G}}$. According to the proof of Lemma \ref{degree-1}, it is easy to see that
$\overline{\mathbb{G}}
=\{G_{0,g},G_{1,g},G_{v,g},G_{2,g},\cdots, G_{\lfloor\frac{g}{2}\rfloor, g},
G_{0,g+1}\}$,
in which the maximum degree of
$G_{1,g},G_{v,g},G_{2,g},\cdots, G_{\lfloor\frac{g}{2}\rfloor,g},
G_{0,g+1}$
equal to $m-g+1$ and $\Delta(G_{0,g})=m-g+2$.
For any $G\in \overline{\mathbb{G}}$,  notice that $m-g+1\leq\Delta(G)\leq m-1$.
By Theorem \ref{outG-0}(ii), we have $q(G)>m-g+2$.
For any $G'\in \underline{\mathbb{G}}$,  notice that
$\Delta(G')\leq m-g$
and $m-g\ge \frac{2m}{3}$ (since $m\geq 3g$),
from Theorem \ref{outG-0}(i)  we have $q(G')\leq m-g+2$ .
Thus each $Q$-index of the graph in $\overline{\mathbb{G}}$ is  more than that of the graph in $\underline{\mathbb{G}}$.
By Theorem \ref{order-girth},
 to complete the proof
it remains to show $q(G_{\lfloor\frac{g}{2}\rfloor,g})>q(G_{0,g+1})$.

Let $\mathbf{x}$ be the Perron vector of $G_{0,g+1}$ corresponding to $q=q(G_{0,g+1})$.
Notice that $m\geq 3g$ and $g\geq 3$. We have $q>m-g+2\geq 8$.
If $g$ is odd,
by symmetry of $G_{0,g+1}$, then
$x_{\lfloor\frac{g}{2}\rfloor}=x_{\lfloor\frac{g}{2}\rfloor+2}$ 
 (see Fig. \ref{g+1}).
Moreover, by the eigenvalue equation of $Q(G_{0,g+1})$, we have
$(q-2)x_{\lfloor\frac{g}{2}\rfloor+1}=
x_{\lfloor\frac{g}{2}\rfloor}+x_{\lfloor\frac{g}{2}\rfloor+2}
=2x_{\lfloor\frac{g}{2}\rfloor}$, and so
$x_{\lfloor\frac{g}{2}\rfloor}\ge x_{\lfloor\frac{g}{2}\rfloor+1}$  since  $q> 8$.
If $g$ is even,
by the symmetry of $G_{0,g+1}$, then
$x_{\lfloor\frac{g}{2}\rfloor}=x_{\lfloor\frac{g}{2}\rfloor+1}$ (see Fig. \ref{g+1}).
Let $$G'=G_{0,g+1}-
\{(\lfloor\frac{g}{2}\rfloor+1)(\lfloor\frac{g}{2}\rfloor+2)\}
+\{\lfloor\frac{g}{2}\rfloor(\lfloor\frac{g}{2}\rfloor+2)\}.$$
Clearly, $G'\cong G_{\lfloor\frac{g}{2}\rfloor,g}$.
By Lemma \ref{lem-perron}, we have
$q(G_{\lfloor\frac{g}{2}\rfloor,g})=q(G')>q(G_{0,g+1})$.

It completes the proof.
\end{proof}

Let $\mathbb{G}_{\Delta}(m, 3)$ be the set of graphs with maximum degree $m-2$ in $\mathbb{G}(m, 3)$. $B_1, B_2, G_{1,3}$ and $G_{v,3}$
are shown in Fig.\ref{graphsg=3}.
By direct observation, we have
$\mathbb{G}_{\Delta}(m, 3)=\{B_1, B_2, G_{1,3}, G_{v,3}\}$.
Thirdly we give the order of $Q$-indices among $\mathbb{G}(m, 3)$ as supplement of Theorem \ref{order-girth} for $g=3$.

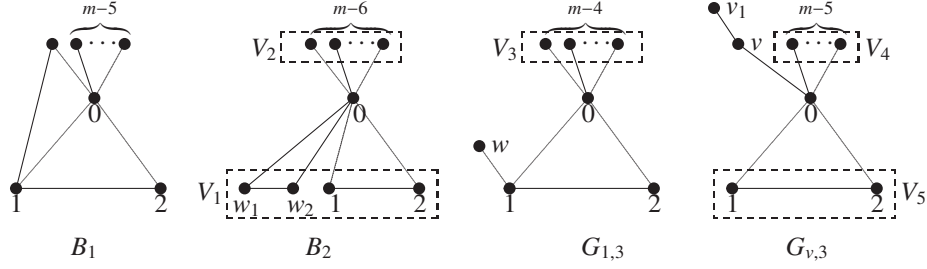
\begin{figure}[h]
  \centering
  \footnotesize
\unitlength 0.8mm 
\linethickness{0.4pt}
\ifx\plotpoint\undefined\newsavebox{\plotpoint}\fi 
\begin{picture}(148,42)(0,0)
\put(57,26){\circle*{2}}
\multiput(50,35)(.033653846,-.043269231){208}{\line(0,-1){.043269231}}
\put(54,35){\line(1,-3){3}}
\multiput(62,35)(-.033557047,-.060402685){149}{\line(0,-1){.060402685}}
\put(50,35){\circle*{2}}
\put(54,35){\circle*{2}}
\put(62,35){\circle*{2}}
\put(56,35){$\ldots$}
\put(57,22){$0$}
\put(50,36){$\overbrace{ \qquad \ \ \ \ }^{m-6}$}
\multiput(57,26)(.0336391437,-.0458715596){327}{\line(0,-1){.0458715596}}
\put(68,11){\circle*{2}}
\put(67,7){$2$}
\multiput(57,26)(-.033613445,-.12605042){119}{\line(0,-1){.12605042}}
\put(53,11){\circle*{2}}
\put(53,11){\line(1,0){15}}
\put(57,26){\line(-2,-3){10}}
\put(57,26){\line(-6,-5){18}}
\put(47,11){\circle*{2}}
\put(39,11){\circle*{2}}
\put(39,11){\line(1,0){8}}
\put(53,7){$1$}
\put(37,7){$w_1$}
\put(46,7){$w_2$}
\put(31,9){$V_1$}
\put(40,33){$V_2$}
\put(49,0){$B_2$}
\put(36,6){\dashbox{1}(35,8)[cc]{}}
\put(45,32){\dashbox{1}(20,5)[cc]{}}
\put(96,26){\circle*{2}}
\multiput(89,35)(.033653846,-.043269231){208}{\line(0,-1){.043269231}}
\put(93,35){\line(1,-3){3}}
\multiput(101,35)(-.033557047,-.060402685){149}{\line(0,-1){.060402685}}
\put(89,35){\circle*{2}}
\put(93,35){\circle*{2}}
\put(101,35){\circle*{2}}
\put(95,35){$\ldots$}
\put(95,22){$0$}
\put(89,36){$\overbrace{ \qquad \ \ \ \ }^{m-4}$}
\multiput(96,26)(-.0336787565,-.0388601036){386}{\line(0,-1){.0388601036}}
\put(83,11){\circle*{2}}
\multiput(96,26)(.0336391437,-.0458715596){327}{\line(0,-1){.0458715596}}
\put(107,11){\circle*{2}}
\put(82,11){\line(1,0){25}}
\put(82,7){$1$}
\put(106,7){$2$}
\put(80,17){$w$}
\put(80,33){$V_3$}
\put(95,0){$G_{1,3}$}
\multiput(78,18)(.033557047,-.046979866){149}{\line(0,-1){.046979866}}
\put(78,18){\circle*{2}}
\put(85,32){\dashbox{1}(20,5)[cc]{}}
\put(133,26){\circle*{2}}
\put(130,35){\line(1,-3){3}}
\multiput(138,35)(-.033557047,-.060402685){149}{\line(0,-1){.060402685}}
\put(130,35){\circle*{2}}
\put(138,35){\circle*{2}}
\put(132,35){$\ldots$}
\put(132,22){$0$}
\put(129,36){$\overbrace{ \qquad \  }^{m-5}$}
\multiput(133,26)(-.0336787565,-.0388601036){386}{\line(0,-1){.0388601036}}
\put(120,11){\circle*{2}}
\multiput(133,26)(.0336391437,-.0458715596){327}{\line(0,-1){.0458715596}}
\put(144,11){\circle*{2}}
\put(119,11){\line(1,0){25}}
\put(119,7){$1$}
\put(143,7){$2$}
\put(148,9){$V_5$}
\put(142,33){$V_4$}
\put(123,34){$v$}
\put(119,40){$v_1$}
\put(129,0){$G_{v,3}$}
\put(117,6){\dashbox{1}(30,8)[cc]{}}
\put(121,35){\line(4,-3){12}}
\put(121,35){\circle*{2}}
\put(117,41){\line(2,-3){4}}
\put(117,41){\circle*{2}}
\put(14,26){\circle*{2}}
\multiput(7,35)(.033653846,-.043269231){208}{\line(0,-1){.043269231}}
\put(11,35){\line(1,-3){3}}
\multiput(19,35)(-.033557047,-.060402685){149}{\line(0,-1){.060402685}}
\put(7,35){\circle*{2}}
\put(11,35){\circle*{2}}
\put(19,35){\circle*{2}}
\put(13,35){$\ldots$}
\put(13,22){$0$}
\put(10,36){$\overbrace{ \qquad  }^{m-5}$}
\multiput(14,26)(-.0336787565,-.0388601036){386}{\line(0,-1){.0388601036}}
\put(1,11){\circle*{2}}
\multiput(14,26)(.0336391437,-.0458715596){327}{\line(0,-1){.0458715596}}
\put(25,11){\circle*{2}}
\put(0,11){\line(1,0){25}}
\put(0,7){$1$}
\put(24,7){$2$}
\put(10,0){$B_1$}
\put(7,35){\line(-1,-4){6}}
\put(127,32){\dashbox{1}(14,5)[cc]{}}
\end{picture}
  \caption{\footnotesize The graphs $B_1$, $B_2$, $G_{1,3}$ and $G_{v,3}$}\label{graphsg=3}
\end{figure}

\begin{proof}[\bf{Proof of Theorem \ref{order-g=3}}]
By Corollary \ref{max-graph-girth},
$G_{0, 3}$  (see Fig.\ref{girth}) attains the maximum $Q$-index  among all graphs in $\mathbb{G}(m, 3)$ for $m\geq 6$.
 Note that $m-1=\Delta(G_{0,3})>\Delta(B_1)=m-2\ge\frac{2m}{3}$.
 By Corollary \ref{compare-max-degree},
we have $q(G_{0,3})>q(B_1)$.

We now prove $q(B_1)>q(B_2)$.
Let $\mathbf{x}$ be the Perron vector of $B_2$ and vertices $w_1$, $w_2$ and $1$ of $B_2$ are shown in Fig.\ref{graphsg=3}.
By the symmetry of graph $Q(B_2)$, we have $x_1=x_{w_1}$.
Let $B'=B_2-\{w_1w_2\}+\{1w_2\}$.
Clearly, $B'\cong B_1$.
By Lemma \ref{lem-perron}, we have $q(B_1)=q(B')>q(B_2)$.

Secondly, we prove $q(B_2)>q(G_{1,3})$.
It is clear that
$B_2$ and $G_{1,3}$ have the equitable partitions $\Pi_1:
V(B_2)=V_1\cup \{0\}\cup V_2$ and $\Pi_2: V(G_{1,3})=\{0\}\cup \{1\}\cup\{2\}\cup\{w\}\cup\{V_3\}$ (see Fig.\ref{graphsg=3}), respectively. Thus the corresponding quotient  matrices are
$$M(B_2)=\begin{pmatrix}
3 & 1&0\\
4& m-2& m-6\\
0 &1&1
\end{pmatrix}, \ \
M(G_{1,3})=\begin{pmatrix}
m-2&1& 1& 0&m-4\\
1&3&1&1&0\\
1&1&2&0&0\\
0&1&0&1&0\\
1&0&0&0&1
\end{pmatrix}.$$
By calculation, the characteristic polynomials of $M(B_2)$ and $M(G_{1,3})$ are
$\varphi(x,B_2)=x^3-(m+2)x^2+(3m-3)x-8$ and $\varphi(x, G_{1,3})=x^5-(m+5)x^4+(6m+3)x^3-(9m-1)x^2+(3m+8)x-4$.
It is easy to verify that
\begin{eqnarray}\label{G1-B2}
\varphi(x, G_{1,3})=\varphi(x,B_2)\cdot (x^2-3x)+(3m-16)x-4.
\end{eqnarray}
Clearly, $\varphi(q(G_{1,3}),G_{1,3})=0$.
Note that $\Delta(G_{1,3})=m-2$.  By Theorem \ref{outG-0}(ii), we have
$q(G_{1,3})>m-1\ge 5$ due to $m\ge6$. Thus $q^2(G_{1,3})-3q(G_{1,3})>0$ and
 $(3m-16)q(G_{1,3})-4>0$.
From (\ref{G1-B2}), we have $\varphi(q(G_{1,3}),B_2)<0$.
Note that
$q(B_2)$ is the largest root of $\varphi(x,B_2)$.
Thus $q(B_2)>q({G_{1,3}})$.

Thirdly, we prove $q({G_{1,3}})>q({G_{v,3}})$.
{$G_{v,3}$} has the equitable partition $\Pi_3: V({G_{v,3}})=\{0\}\cup\{v\}\cup \{v_1\} \cup \{V_4\}\cup \{V_5\} $  (see Fig.\ref{graphsg=3}). Thus the quotient matrix with respect to  $\Pi_3$ is
$$M({G_{v,3}})=\begin{pmatrix}
m-2&1& 0&m-5 &2\\
1&2&1&0&0\\
0&1&1&0&0\\
1&0&0&1&0\\
1&0&0&0&3
\end{pmatrix}.$$
By calculation, the characteristic polynomial of $M({G_{v,3}})$ is
$\varphi(x,{G_{v,3}})
=x^5-(m+5)x^4+(6m+4)x^3-(10m-2)x^2+(3m+12)x-4$.
Thus
\begin{eqnarray}\label{G1-Gv}
\varphi(x, {G_{v,3}})=\varphi(x, {G_{1,3}}) +x(x^2-(m-1)x+4).
\end{eqnarray}
Recall that $\Delta({G_{v,3}})=m-2$. By Theorem \ref{outG-0}(ii), we have
$q({G_{v,3}})>m-1$. Thus $q^2({G_{v,3}})-(m-1)q({G_{v,3}})+4>0$.
Clearly, $\varphi(q({G_{v,3}}), {G_{v,3}})=0$. From (\ref{G1-Gv}), we get $\varphi(q({G_{v,3}}), {G_{1,3}})<0$.
Note that $q({G_{1,3}})$ is the largest root of $\varphi(x,  {G_{1,3}})$. Thus $q({G_{v,3}})<q({G_{1,3}})$.
 For any $G\in \mathbb{G}(m, 3) \setminus \{\mathbb{G}_{\Delta}(m, 3)\cup G_{0,3}\}$, we have
$\Delta(G)\leq m-3$.
Notice that $m-2=\Delta({G_{v,3}})> m-3 \geq\frac{2m}{3}$ due to $m\geq 9$, by Corollary
\ref{compare-max-degree} we have $q{G_{v,3}})>q(G)$.
It completes the proof.
\end{proof}

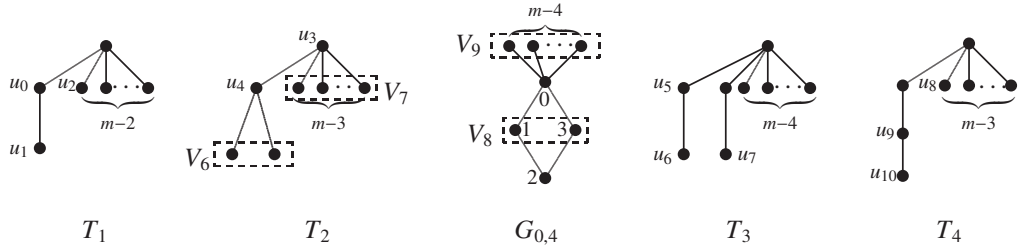
\begin{figure}
  \centering
  \footnotesize
\unitlength 0.8mm 
\linethickness{0.65pt}
\ifx\plotpoint\undefined\newsavebox{\plotpoint}\fi 
\begin{picture}(163.414,34)(0,0)
\put(45,0){$T_2$}
\put(80,0){$G_{0,4}$}
\put(115,0){$T_{3}$}
\put(150,0){$T_{4}$}
\put(8,0){$T_1$}
\put(12,32){\circle*{2}}
\multiput(12,32)(-.033613445,-.058823529){119}{\line(0,-1){.058823529}}
\put(8,25){\circle*{2}}
\put(12,32){\line(0,-1){7}}
\put(12,25){\circle*{2}}
\put(12,32){\line(1,-1){7}}
\put(19,25){\circle*{2}}
\multiput(12,32)(-.052884615,-.033653846){208}{\line(-1,0){.052884615}}
\put(1,25){\circle*{2}}
\put(1,25){\line(0,-1){10}}
\put(1,15){\circle*{2}}
\put(48,32){\circle*{2}}
\multiput(48,32)(-.033613445,-.058823529){119}{\line(0,-1){.058823529}}
\put(44,25){\circle*{2}}
\put(48,32){\line(0,-1){7}}
\put(48,25){\circle*{2}}
\put(48,32){\line(1,-1){7}}
\put(55,25){\circle*{2}}
\multiput(48,32)(-.052884615,-.033653846){208}{\line(-1,0){.052884615}}
\put(37,25){\circle*{2}}
\multiput(37,25)(.03370787,-.12359551){89}{\line(0,-1){.12359551}}
\put(40,14){\circle*{2}}
\multiput(37,25)(-.033613445,-.092436975){119}{\line(0,-1){.092436975}}
\put(33,14){\circle*{2}}
\put(30,12){\dashbox{1}(13,4)[cc]{}}
\put(42,23){\dashbox{1}(15,4)[cc]{}}
\put(85,26){\circle*{2}}
\put(84,22){\scriptsize$0$}
\put(87,17){\scriptsize$3$}
\put(81,17){\scriptsize$1$}
\put(82,9){\scriptsize$2$}
\put(85,32){$\ldots$}
\put(79,33){$\overbrace{ \qquad  \ \ \ }^{m-4}$}
\put(70,31){$V_9$}
\put(72,16){$V_8$}
\multiput(85,26)(.033557047,-.053691275){149}{\line(0,-1){.053691275}}
\multiput(85,26)(-.033557047,-.053691275){149}{\line(0,-1){.053691275}}
\multiput(80,18)(.033557047,-.053691275){149}{\line(0,-1){.053691275}}
\multiput(90,18)(-.033557047,-.053691275){149}{\line(0,-1){.053691275}}
\put(80,18){\circle*{2}}
\put(90,18){\circle*{2}}
\put(85,10){\circle*{2}}

\put(79,32){\line(1,-1){6}}
\put(83,32){\line(1,-3){2}}
\put(91,32){\line(-1,-1){6}}
\put(79,32){\circle*{2}}
\put(83,32){\circle*{2}}
\put(91,32){\circle*{2}}
\put(76,30){\dashbox{1}(18,4)[cc]{}}
\put(78,16){\dashbox{1}(14,4)[cc]{}}
\put(122,32){\circle*{2}}
\multiput(122,32)(-.033613445,-.058823529){119}{\line(0,-1){.058823529}}
\put(118,25){\circle*{2}}
\put(122,32){\line(0,-1){7}}
\put(122,25){\circle*{2}}
\put(122,32){\line(1,-1){7}}
\put(129,25){\circle*{2}}
\put(122,32){\line(-1,-1){7}}
\put(122,32){\line(-2,-1){14}}
\put(115,25){\circle*{2}}
\put(108,25){\circle*{2}}
\put(115,25){\line(0,-1){11}}
\put(115,14){\circle*{2}}
\put(108,25){\line(0,-1){11}}
\put(108,14){\circle*{2}}
\put(155.414,32.414){\circle*{2}}
\multiput(155.414,32.414)(-.033613445,-.058823529){119}{\line(0,-1){.058823529}}
\put(151.414,25.414){\circle*{2}}
\put(155.414,32.414){\line(0,-1){7}}
\put(155.414,25.414){\circle*{2}}
\put(155.414,32.414){\line(1,-1){7}}
\put(162.414,25.414){\circle*{2}}
\multiput(155.414,32.414)(-.052884615,-.033653846){208}{\line(-1,0){.052884615}}
\put(144.414,25.414){\circle*{2}}
\put(144.414,25.414){\line(0,-1){8}}
\put(144.414,17.414){\circle*{2}}
\put(144.414,17.414){\line(0,-1){7}}
\put(144.414,10.414){\circle*{2}}
\put(13,25){$\ldots$}
\put(8,24){$\underbrace{ \qquad  \ \ \ }_{m-2}$}
\put(4,25){\scriptsize$u_2$}
\put(-4,25){\scriptsize$u_0$}
\put(-4,14){\scriptsize$u_1$}
\put(50,25){$\ldots$}
\put(43,24){$\underbrace{ \qquad  \ \ \ }_{m-3}$}
\put(44,33){\scriptsize$u_3$}
\put(32,25){\scriptsize$u_4$}
\put(25,12){$V_6$}
\put(58,23){$V_7$}
\put(123,25){$\ldots$}
\put(118,24){$\underbrace{ \qquad  \ \ \ }_{m-4}$}
\put(103,25){\scriptsize$u_5$}
\put(103,13){\scriptsize$u_6$}
\put(117,13){\scriptsize$u_7$}
\put(157,25){$\ldots$}
\put(151,24){$\underbrace{ \qquad  \ \ \ }_{m-3}$}
\put(147,25){\scriptsize$u_8$}
\put(140,17){\scriptsize$u_9$}
\put(139,10){\scriptsize$u_{10}$}
\end{picture}

  \caption{\footnotesize Some graphs used in the  Proof of Theorem \ref{order-m}}\label{some-graph}
\end{figure}

Notice that $\mathbb{G}_m$ is the union of $\cup_{g\geq3}\mathbb{G}(m, g)$ and some trees. The ordering of $Q$-indices of graphs in $\mathbb{G}(m, g)$
and $\mathbb{G}(m, \geq g)$ would induce related result among graphs in $\mathbb{G}_m$. At last we prove Theorem \ref{order-m} that orders of the first eleven
largest $Q$-indices among $\mathbb{G}_m$.

\begin{proof}[\bf{Proof of Theorem \ref{order-m}}]
Denote by $\mathbb{G}_{\Delta}(m)$ and $\mathbb{G}_{\leq \Delta}(m)$  the set of  graphs in $\mathbb{G}_m$  with maximum degree $\Delta$ and no more than
$\Delta$, respectively.
Clearly, $\mathbb{G}_{m}(m)=\{K_{1,m}\}$,  $\mathbb{G}_{m-1}(m)=\{G_{0,3},T_1\}$ and
$\mathbb{G}_{m-2}(m)=\{B_1, B_2, G_{1, 3}, G_{v, 3}, T_{2}, G_{0, 4}, T_{3}, T_{4}\}$.
Moreover,
 we can partition
 $$\mathbb{G}_m=\mathbb{G}_m(m)\cup \mathbb{G}_{m-1}(m)\cup \mathbb{G}_{m-2}(m)\cup \mathbb{G}_{\leq m-3}(m).$$

Note that $\Delta(G_{0,3})=m-1\geq \frac{2m}{3}$ since $m\geq 9$.
By Theorem\ref{outG-0} (i), we have
$q(G_{0,3})\leq m+1$.
If  $q(G_{0,3})=m+1$, then by the eigenvalue equation of $Q(G_{0,3})$,   we  deduce that $m=-2$ or $3$,  which contradicts  $m\geq 9$.
Thus
$q(G_{0,3})< m+1=q(K_{1,m})$.

Secondly, we will show $q(G_{0,3})>q(T_1)$.
Let $\mathbf{x}$ be the Perron vector of $T_1$.
By the eigenvalue equation of $Q(T_1)$, we have $x_{u_2}=\frac{q^2(T_1)-3q(T_1)+1}{q(T_1)-1}x_{u_1}$.
Notice that $\Delta(T_1)= m-1$. By Lemma \ref{outG-0}(ii),
 we have $q(T_1)> q(K_{1,m-1})=m\ge 9$. Thus $\frac{q^2(T_1)-3q(T_1)+1}{q(T_1)-1}>1$ and so $x_{u_2}>x_{u_1}$.
 Let $G'=T_1-u_0u_1+u_0u_2$.
 Clearly, $G'\setminus\{u_1\}\cong G_{0,3}$.
 By Lemma \ref{lem-perron}, we have $q(G_{0,3})=q(G')>q(T_1)$.

Again notice that $m-1=\Delta(T_{1})>\Delta(B_1)=m-2\geq\frac{2m}{3}$ since $m\geq 9$, we have
$q(T_{1})>q(B_1)$ by Corollary \ref{compare-max-degree}.
Thus,  from Theorem \ref{order-g=3} we obtain
$$q(K_{1,m})>q(G_{0,3})>q(T_1)>q(B_1)>q(B_2)>q(G_{1,3})>q(G_{v,3}).$$

Thirdly, we will show $q(G_{v,3})>q(T_2)>q(G_{0, 4})$  by comparing the largest root of their quotient matrices of two graphs.
As shown in Fig. \ref{some-graph},
$T_2$ and  $G_{0,4}$ respectively have the equitable partition $\Pi_4:  V(T_2)=\{u_3\}\cup \{u_4\}\cup V_6\cup V_7$ and  $\Pi_5: V(G_{0,4})=\{0\}\cup V_8\cup\{2\}\cup V_9$.
 The quotient matrices with respect to $\Pi_4$ and  $\Pi_5$ are respectively given by
$$M(T_2)=\begin{pmatrix}
m-2& 1& 0& m-3\\
1& 3 &2& 0\\
0& 1 &1 &0\\
1 &0& 0& 1
\end{pmatrix} \ \ \mbox{and} \ \
M(G_{0,4})=\begin{pmatrix}
m-2& 2& 0& m-4\\
1& 2 &1& 0\\
0& 2 &2 &0\\
1 &0& 0& 1
\end{pmatrix}.
$$
By calculation, the characteristic polynomials of $M(T_2)$ and $M(G_{0, 4})$  are respectively
$\varphi(x,T_2)= x^4 - ( m +3)x^3 + (4m - 3)x^2 - ( m + 1)x$ and
$\varphi(x, G_{0, 4})= x^4 - ( m + 3)x^3 + (4m - 2)x^2 - 2mx$.
Combining (\ref{G1-Gv}),
we obtain that
\begin{eqnarray}\label{T2-Gv3}
\varphi(x,G_{v,3})=(x-2)\cdot\varphi(x,T_2)+\frac{\varphi(x,T_2)}{x}+h(x),
\end{eqnarray}
where $h(x)=(13 - 3m)x + m - 3$.
Since $m\ge9$, we have $h(x)$ is decrease.
Note that $\Delta(T_2)= m-2\le m-1$, by Lemma \ref{outG-0}, we have $q(T_2)>q(K_{1,m-2})= m-1$ and then
$h(q(T_2))<h(m-1)=- 3m^2 + 17m - 16<0$ due to $m\ge9$.
From (\ref{T2-Gv3}), we get $\varphi(q(T_2),G_{v,3})<0$.
Note that $q(G_{v,3})$ is the largest root of $\varphi(x,G_{v,3})$.  Thus $q(G_{v,3})>q(T_2)$.
On the other hand,
one also can  verify  that
\begin{eqnarray}\label{eq-10}\varphi(x,G_{0, 4})=\varphi(x,T_2)+x(x + (1 - m)).\end{eqnarray}
Recall that  $q(T_2)>m-1$,
we have $\varphi(q(T_2),G_{0,4})>0$.
Since $q(G_{0,4})$ is the largest root of $\varphi(x,G_{0,4})$, we have $q(T_2)>q(G_{0,4})$ and thus $q(G_{v,3})>q(T_2)>q(G_{0,4})$.

Fourthly,  we show $q(G_{0,4})>q(T_3)$.
Let $\mathbf{y}$ be the Perron vector of $Q(T_3)$.
One can see $y_{u_6}=y_{u_7}$ from the symmetry of $T_3$ (see Fig.\ref{some-graph}).
 Let $G''=T_3-u_5u_6+u_5u_7$.
 Clearly, $G''\setminus\{u_6\}\cong G_{0,4}$.
 By Lemma \ref{lem-perron}, we have $q(G_{0,4})=q(G'')>q(T_3)$.

Now we will show $q(T_{3})>q(T_4)$.
 Let $\mathbf{z}$ be the Perron vector of $Q(T_4)$.
One can verify $z_{u_8}=\frac{q^3(T_4) - 5q^2(T_4) + 6q(T_4) - 1}{(q(T_4)-1)^2}z_{u_9}$ from the eigenvalue equation of $Q(T_4)$.
 Notice that  $q(T_4)>q(K_{1, m-2})=m-1\ge 8$ due to $m\ge9$. We have  $\frac{q^3(T_4) - 5q^2(T_4) + 6q(T_4) - 1}{(q(T_4)-1)^2}>1$ and so $z_{u_8}>z_{u_9}$.
Let $G'''=T_4-u_{10}u_9+u_{10}u_8$.
 Clearly, $G'''\cong T_{3}$.
 By Lemma \ref{lem-perron}, we have $q(T_{3})=q(G''')>q(T_4)$.

For any $G\in \mathbb{G}_{\leq m-3}(m)$, we have $\Delta(T_4)=m-2>m-3\geq \Delta(G)\geq \frac{2m}{3}$ since $m\geq 9$.
By Corollary \ref{compare-max-degree}, we get $q(T_{4})>q(G)$.
Therefore,  by the above discussion, the first eleven largest $Q$-indices of graphs in $\mathbb{G}_m$ are given by
\begin{eqnarray*}q(K_{1,m})&>&q(G_{0,3})>q(T_1)>q(B_1)>q(B_2)>q(G_{1,3})>q(G_{v,3})
>q(T_2)>q(G_{0,4})\\&>&q(T_3)>q(T_4).\end{eqnarray*}
It completes the proof.
\end{proof}



\begin{thebibliography}{11}{\small
\bibitem{Brouwer}
A. E. Brouwer, W. H. Haemers, Spectra of Graphs, Springer, Berlin. 2011.
\vspace{-0.3cm}
\bibitem{Chang}
 A. Chang, Q.X. Huang, Ordering trees by their largest eigenvalues, Linear Algebra Appl. 370 (2003) 175-184.
\vspace{-0.3cm}

\bibitem{Cvetkovic-direct}
 D. Cvetkovi\'{c}, Some possible directions in further investigations of graph spectra, Algebra Methods in Graph Theory, vol. 1, North-Holland, Amsterdam, 1981, pp. 47-67.
\vspace{-0.3cm}
\bibitem{C2}
  D. Cvetkovi\'{c}, P. Rowlinson, S.K. Simi\'{c},
 Signless Laplacians of finite graphs, Linear Algebra Appl. 423 (2012) 155-171.
   \vspace{-0.3cm}

\bibitem{Cvetkovic1}  D. Cvetkovi\'{c}, S.K. Simi\'{c}, Towards a spectral theory of graphs based on the signless Laplacian I, Publ. Inst. Math.  85(99) (2009) 19-33.
\vspace{-0.3cm}

\bibitem{Cvetkovic2}  D. Cvetkovi\'{c}, S.K. Simi\'{c}, Towards a spectral theory of graphs based on the signless Laplacian II, Linear Algebra Appl. 432 (2010) 2257-2272.
   \vspace{-0.3cm}

\bibitem{Cvetkovic3}  D. Cvetkovi\'{c}, S.K. Simi\'{c}, Towards a spectral theory of graphs based on the signless Laplacian III, Appl. Anal. Discrete Math. 4 (2010) 156-166.
   \vspace{-0.3cm}

\bibitem{C.M}	
M.Z. Chen, A.M. Liu, X.D. Zhang, Spectral extremal results on the $\alpha$-index of graphs without minors and star forests, 2022, arXiv:2204.00181.\vspace{-0.3cm}
\bibitem{Zhai}
   W.W. Chen, B. Wang, M.Q. Zhai, Signless Laplacian spectral radius of graphs without short cycles or long cycles,
Linear Algebra Appl. 645 (2022) 123-136.
   \vspace{-0.3cm}

\bibitem{M.C}
M.Z. Chen, X.D. Zhang, On the signless Laplacian spectral radius of $K_{s,t}$-minor free graphs, Linear Multilinear Algebra 69 (10) (2021) 1922-1934. \vspace{-0.3cm}
\bibitem{M.F}
M.A.A. de Freitas, V. Nikiforov, L. Patuzzi, Maxima of the $Q$-index: forbidden $4$-cycle and $5$-cycle, Electron. J. Linear Algebra 26 (2013) 905-916.\vspace{-0.3cm}

\bibitem{Feng}
 L.H. Feng, G.H. Yu, On three conjectures involving the signless Laplacian spectral radius of graphs,
Publ. Inst. Math.  85 (99) (2009) 35-38.
\vspace{-0.3cm}

\bibitem{guo}
J.M. Guo, On the Laplacian spectral radius of trees with fixed diameter, Linear Algebra Appl. 419 (2006) 618-629.\vspace{-0.3cm}

\bibitem{Feng-L}
 L.H. Feng, G.H. Yu, The signless Laplacian spectral radius of graphs with given diameter, Util. Math. 83 (2010) 265-276.
 \vspace{-0.3cm}

\bibitem{Jia-Li}
  H.M. Jia, S.C. Li, S.J. Wang, Ordering the maxima of $L$-index and $Q$-index: Graphs with given size and diameter, Linear Algebra Appl. 652 (2022) 18-36.
\vspace{-0.3cm}

\bibitem{He-B}
 B. He,  Y.L. Jin,  X.D. Zhang,
Sharp bounds for the signless Laplacian spectral radius in terms of clique number,
Linear Algebra Appl. 438 (2013) 3851-3861.
\vspace{-0.3cm}

\bibitem{Hong}
    Y. Hong, X.D. Zhang, Sharp upper and lower bounds for largest eigenvalue of the Laplacian
    matrices of trees, Discrete Math. 296 (2005) 187-197.
    \vspace{-0.3cm}

\bibitem{Liu-mh}
 M.H. Liu, K.Ch. Das,  H.J. Lai,
The (signless) Laplacian spectral radii of $c$-cyclic graphs with $n$ vertices, girth $g$ and $k$ pendant vertices,
Linear Multilinear Algebra 65 (2017) 869-881.
\vspace{-0.3cm}


\bibitem{Liu-liu}
 M.H. Liu, B.L. Liu, B. Cheng,
Ordering (signless) Laplacian spectral radii with maximum degrees of graphs,
Discrete Math. 338 (2015) 159-163.
\vspace{-0.3cm}

\bibitem{Lou-Z}
  Z.Z. Lou, J.M. Guo, Z.W. Wang, Maxima of $L$-index and $Q$-index: graphs with given size and diameter, Discrete Math. 344 (2021) 112533. \vspace{-0.3cm}
\bibitem{Lin}
 S.T. Liu, H.Q. Lin, J.L. Shu, A note on (signless) Laplacian spectral ordering with maximum degrees of graphs, Linear Algebra Appl. 521 (2017) 135-141.
\vspace{-0.3cm}

\bibitem{wang-uni-bi}
K. Li, L.G. Wang, G.P. Zhao,
The signless Laplacian spectral radius of unicyclic and bicyclic graphs with a given girth,
Electron. J. Combin. 18 (2011) Paper 183.
\vspace{-0.3cm}

\bibitem{Nikiforov-2}
V. Nikiforov, A spectral condition for odd cycles in graphs, Linear Algebra Appl. 428 (2008) 1492-1498.\vspace{-0.3cm}

\bibitem{wang-tri}
 L. Qiao, L.G. Wang,
The signless Laplacian spectral radius of tricyclic graphs with a given girth,
J. Math. Res. Appl. 34 (2014) 379-391.
\vspace{-0.3cm}

\bibitem{turan}
P. Tur\'{a}n, On an extremal problem in graph theory, Mat. Fiz. Lapok, 48 (1941) 436-452.\vspace{-0.3cm}

\bibitem{Wang-J}
J.F. Wang, Q.X. Huang, Maximizing the signless Laplacian spectral radius of graphs with given diameter or cut vertices, Linear Multilinear Algebra 59 (2011) 733-744.
 \vspace{-0.3cm}

\bibitem{Wei-Liu}
 F.Y. Wei,  M.H. Liu,
Ordering of the signless Laplacian spectral radii of unicyclic graphs,
Australas. J. Combin. 49 (2011) 255-264.
\vspace{-0.3cm}

\bibitem{Yu-G}
 G.L. Yu, Y.R.  Wu, J.L. Shu,
Signless Laplacian spectral radii of graphs with given chromatic number,
Linear Algebra Appl. 435 (2011) 1813-1822.
\vspace{-0.3cm}

\bibitem{Yuan-Liu}
X.Y. Yuan, Y. Liu, M.M. Han, The Laplacian spectral radius of trees and maximum vertex degree,
Discrete Math. 311 (2011) 761-768.
\vspace{-0.3cm}

\bibitem{Yuan}
 X.Y. Yuan, H.Y. Shan, Y. Liu, On the Laplacian spectral radii of trees, Discrete Math. 309 (2009) 4241-4246.
\vspace{-0.3cm}

\bibitem{zhai-xue-lou}
   M.Q. Zhai, J. Xue, Z.Z. Lou, The signless Laplacian spectral radius of graphs with a prescribed
number of edges, Linear Algebra Appl. 603 (2020) 154-165.
\vspace{-0.3cm}


\bibitem{zhai-1}
M.Q. Zhai, B. Wang, Proof of a conjecture on the spectral radius of  $C_4$-free graphs, Linear Algebra Appl. 437 (2012) 1641-1647.\vspace{-0.3cm}

\bibitem{zhai-lin-1}
M.Q. Zhai, H.Q. Lin, Spectral extrema of graphs: forbidden hexagon, Discrete Math. 343 (2020) 112028.\vspace{-0.3cm}

\bibitem{zhai-xue}
M.Q. Zhai, J. Xue, R.F. Liu,
A spectral extremal problem on graphs with given size and matching.
Linear Multilinear Algebra  (2021)
DOI: 10.1080/03081087.2021.1915231.
 \vspace{-0.3cm}

\bibitem{zhang-xd}
 X.D. Zhang,
The signless Laplacian spectral radius of graphs with given degree sequences,
Discrete Appl. Math. 157 (2009) 2928-2937.
\vspace{-0.3cm}

\bibitem{zhuzx}
 Z.X. Zhu, The signless Laplacian spectral radius of bicyclic graphs with a given girth, Electron. J. Linear Algebra 22 (2011) 378-388.
\vspace{-0.3cm}




}
\end{thebibliography}
\end{document}